\documentclass[a4paper,11pt]{article}
\usepackage[utf8]{inputenc}
\usepackage[T1]{fontenc}
\usepackage{graphicx}
\usepackage{amsmath,amsfonts,amssymb,amsthm}
\usepackage{mathtools}
\usepackage{mathrsfs}
\usepackage{bm}
\usepackage{color}
\usepackage{fullpage}
\usepackage{bbold}

\definecolor{myred}{rgb}{0.77, 0.0, 0.1}
\definecolor{crimson}{rgb}{0.86, 0.08, 0.24}
\definecolor{awesome}{rgb}{1.0, 0.13, 0.32}
\definecolor{newgreen}{rgb}{0.0,0.6,0.0}
\definecolor{malachite}{rgb}{0.04, 0.85, 0.32}
\definecolor{pastelgreen}{rgb}{0.47, 0.87, 0.47}
\definecolor{myturq}{rgb}{0.1, 0.7, 0.7}

\usepackage{hyperref}
\hypersetup{
  colorlinks,
  linkcolor=blue, %
  citecolor=blue, %
  linktoc=all %
}

\usepackage{relsize}

\renewcommand{\leq}{\leqslant}

\renewcommand{\le}{\leqslant}
\renewcommand{\ge}{\geqslant}
\renewcommand{\hat}{\widehat}

\newcommand{\di}{\mathrm{d}}

\newcommand{\eps}{\varepsilon}

\newcommand{\R}{\mathbb{R} }

\newcommand{\med}{\mathrm{med}}
\newcommand{\Med}{\mathrm{Med}}

\renewcommand{\ln}{\log}
\usepackage{xspace}

\newcommand{\lint}{[\![}
\newcommand{\rint}{]\!]}

\renewcommand{\P}{\mathbb P}

\newcommand{\kl}{\mathrm{KL}}%
\newcommand{\kll}[2]{\kl ({#1}, {#2})}%

\newcommand{\dtv}{D_{\mathrm{TV}}}

\newcommand{\gaussdist}{\mathcal{N}}

\newtheorem{proposition}{Proposition}%
\newtheorem{theorem}{Theorem}
\newtheorem{lemma}{Lemma}
\newtheorem{corollary}{Corollary}

\theoremstyle{definition}

\theoremstyle{remark}

\title{Performance of the empirical median for location estimation in heteroscedastic settings
}

\author{Sirine Louati\footnote{CREST, ENSAE, Palaiseau, France;  \tt{sirine.louati@ensae.fr}} 
}
\date{\today}

\begin{document}

\maketitle

\begin{abstract}
We investigate the performance of the empirical median for location estimation in heteroscedastic settings. 
Specifically, we consider independent symmetric real-valued random variables that share a common but unknown location parameter while having different and unknown scale parameters. 
Estimation under heteroscedasticity arises naturally in many practical situations and has recently attracted considerable attention. 
In this work, we analyze the empirical median as an estimator of the common location parameter and derive matching non-asymptotic upper and lower bounds on its estimation error. 
These results fully characterize the behavior of the empirical median in heteroscedastic settings, clarifying both its robustness and its intrinsic limitations and offering a precise understanding of its performance in modern settings where data quality may vary across sources.  

\end{abstract}

\tableofcontents

\section{Introduction}
\label{sec:introduction}
Let \((X_1, \dots, X_n)\) be independent  real-valued random variables symmetric around a common location parameter \(\theta\) but with potentially different and unknown scales \(\sigma_i > 0\), for \(i \in \{1, \dots, n\}\). Without loss of generality, we suppose that the scales are ordered as \(\sigma_1 \leq \cdots \leq \sigma_n\), although we focus on estimators that are invariant to permutations of the indices and do not exploit this ordering.
In this heteroscedastic setting, our goal is to estimate \(\theta\) when only a single observation is available from each \(X_i\), and no direct information on the \(\sigma_i\)'s is accessible. While several recent works, detailed in Section \ref{sec:discussion}, have proposed estimators that outperform the empirical median~\cite{compton2025attainability}, by leveraging assumptions such as local modality \cite{diakonikolas2020outlier, lugosi2019mean} or carefully aggregating the most informative observations, this paper takes a different angle. Rather than competing with these methods, our aim is to sharpen the theoretical understanding of the empirical median as a natural benchmark procedure.

Two natural estimators that one may consider at first glance are the empirical mean and the maximum likelihood estimator (MLE). In particular, in the Gaussian case, one immediately obtains explicit results regarding their estimation errors. Indeed, 
in this framework, namely when $X_i \sim \gaussdist (\theta,\sigma_i^2),$ $i=1,\dots,n$, a straightforward, though potentially suboptimal, approach consists in using the empirical mean estimator, defined by
\[ \hat{\theta}= \frac{1}{n} \sum_{i=1}^{n} X_i ,\] and following a Gaussian distribution with mean $\theta$ and variance $\sigma^2= \frac{1}{n^2}\sum_{i=1}^{n}\sigma^2_{i}$.
For $\delta \in (0,1)$,  we have for this estimator that with probability at least equal to $1-\delta$,
\begin{align}
\label{eq:mean}
    \lvert \hat{\theta}-\theta \rvert \le \frac{\sqrt{2\left(\sum_{i=1}^{n}\sigma^2_{i}\right)\ln{\left(\frac{1}{\delta}\right)}}}{n}.
\end{align}
It can be observed that this estimator is highly sensitive to observations with high variances, yet relatively insensitive to ones with small variances that contain most information about the mean. In fact, if we fix the first $n-1$ variances and let $\sigma_n$ tend to infinity, Bound (\ref{eq:mean}) increases in this case. Conversely, if we assume that the first standard deviation $\sigma_1$ tends to zero, thereby implying a high degree of certainty regarding the observation in question, we observe that Bound (\ref{eq:mean}) remains relatively stable.

Another natural estimator in this setting, and if the $\sigma_i$'s are known, is the maximum likelihood estimator which is defined as 
\[ \hat{\theta}_{MLE}= \frac{\sum_{i=1}^{n} \sigma_i^{-2}X_i}{\sum_{i=1}^{n} \sigma_i^{-2}}  ,\] 
and which also follows a Gaussian distribution with mean $\theta$ and variance $\sigma^2= \Big(\sum_{i=1}^{n} \sigma_i^{-2}\Big)^{-1}$. For $\delta \in (0,1)$, we have with probability at least equal to $1-\delta$ that
\begin{align}
\label{eq:mle}
\lvert \hat{\theta}_{MLE} -\theta\rvert \le  \sqrt{\frac{2\ln{(\frac{1}{\delta})}}{\sum_{i=1}^{n} \sigma_i^{-2}}}.
\end{align}
These two previous results are based 
a classical tail bound that can be found in \cite[p.~9]{ledoux2013probability} which gives that if $X \sim \gaussdist (0, 1)$, then for all $t>0$, $ \P\left(\lvert X \rvert \ge t\right) \le \exp\left(-\frac{t^2}{2}\right).$

Consequently, this estimator is highly responsive to minor fluctuations. 
Indeed, if the  smallest standard deviation $\sigma_{1}$ tends to zero, Bound (\ref{eq:mle}) also tends to zero. Furthermore, the estimator is not sensitive to observations with large variances, as demonstrated by the fact that when $\sigma_{n}$ tends to $+\infty$, Bound (\ref{eq:mle}) remains bounded.
Moreover, and in this Gaussian framework, the maximum likelihood estimator is minimax in the sense of the squared error
\cite[Section 16.4 p.~330]{keener2010theoretical}.
However, this estimator is based on the assumption of prior knowledge of the variances, which is not applicable in our case.

Our main goal in this paper is to
provide matching non-asymptotic upper and lower bounds on the deviation of the empirical median. 
This contributes to a precise performance characterization of a simple, natural, and widely-used estimator. The empirical median may indeed be suboptimal in various regimes, particularly when the signal is concentrated in the lowest-variance observations, but its robustness, simplicity, and parameter-free nature make it a relevant object of study in its own right.

\paragraph{Notations}
Throughout the paper, we use the following notation. We denote $M:= \Med(X)$ the theoretical median of a random variable $X$ and $\hat{\med}(X_1,\dots , X_n)$ the empirical median. Moreover, $\lfloor .  \rfloor$ denotes the floor function. 

\paragraph{Structure of the paper}The paper is organized as follows. We will first analyze the empirical median estimator in Section \ref{sec:median} by establishing an upper bound and a lower one that matches the upper bound within a numerical constant. In Section \ref{sec:discussion}, we first expose prior works and recent results and then show that our upper bound result improves and generalizes two of the previous works. Finally, in Section \ref{sec:proofs}, we give the proofs of the established results.

\section{Analysis of the empirical median estimator: Main results }
\label{sec:median}

Given the symmetry of the problem, the empirical median estimator is a natural one to study. Moreover, unlike the mean, the median is relatively unaffected by high variances.
In this section, the objective is to control the probability of deviation of the median and to obtain upper and lower bounds on the empirical median in high probability.

\subsection{Upper bound}
We state the theorem that determines the upper bound of the estimation error of the empirical median.
\begin{theorem}
  \label{prop:Upper-bound-median}
  
  Let $\delta \in (0,1)$, $t \ge 0$ and $C$ a positive constant such that $0<C<1$ and $\delta > \exp{\left(-2nC^2\right)}$. Let $(X_1, \dots ,  X_n)$ be $n$ independent
  random variables that are symmetric around the location parameter $\theta$ and such that for all $i \in \lint 1,n \rint$ there exists $\sigma_i$ such that, if $0 \le t \le \sigma_i$, then
  \[ \P(\theta \le X_i \le \theta+t) \ge C \frac{t}{\sigma_i}.\]
  Then, we have that with probability at least equal to $1-\delta$, 
    \[ \lvert \hat{\med}(X_1,\dots , X_n) -\theta\rvert \le \frac{1}{C\sqrt{2}}\frac{\sqrt{n\ln{(\frac{2}{\delta})}}}{\sum_{i=j+1}^{n}\sigma_{i}^{-1}} ,\]
    with $j= \lfloor \frac{1}{C\sqrt{2}} \sqrt{n\ln{(\frac{1}{\delta})}} \rfloor$.
\end{theorem}
\paragraph{Remark} 
Note that the condition in Theorem \ref{prop:Upper-bound-median} is satisfied if $\frac{X_i}{\sigma_i}$ has a bounded density on some interval around $\theta$ and does not require the assumption of symmetry. This means that the condition of symmetric random variables around $\theta$ such that $\P(\theta \le X_i \le \theta+t) \ge C \frac{t}{\sigma_i}$ can be replaced by the condition that $\theta$ is a median of $X$ and that
\[ \min \left(\P(\theta \le X_i \le \theta+t), \P(\theta-t \le X_i \le \theta) \right) \ge C \frac{t}{\sigma_i}.\]
Note that by the above condition, $\theta$ is in fact the unique median of $X$.

The following corollary is an application of Theorem \ref{prop:Upper-bound-median} in the Gaussian setting.
\begin{corollary}
  \label{cor:Upper-bound-median}
  
  Let $\delta \in (0,1)$ such that $\delta > \exp{\left(-\frac{n}{e\pi}\right)}$ and let $(X_1, \dots ,  X_n)$ be $n$ independent 
  Gaussian random variables where $X_i \sim \gaussdist (\theta,\sigma_i^2)$. Then, with probability at least $1-\delta$ we have 
    \[ \lvert \hat{\med}(X_1,\dots , X_n) -\theta\rvert \le 2.93\frac{\sqrt{n\ln{(\frac{2}{\delta})}}}{\sum_{i=j+1}^{n}\sigma_{i}^{-1}} ,\]
    with $j= \lfloor 2.93 \sqrt{n\ln{(\frac{1}{\delta})}} \rfloor$. 
\end{corollary}
This result follows from Theorem \ref{prop:Upper-bound-median} with $ C= \frac{\exp{(-\frac{1}{2})}}{\sqrt{2\pi}}=0.24$.

\subsection{Lower bound}
As discussed in Section \ref{sec:introduction}, no estimator can have an estimated error that is smaller than that of the maximum likelihood estimator, which is aware of all the variances and whose bound is minimax. However, in the absence of knowledge of these variances, the problem becomes significantly more challenging. It would therefore be of interest to identify a lower bound that would be larger than $ \sqrt{\frac{2\ln{(\frac{1}{\delta})}}{\sum_{i=1}^{n} \sigma_i^{-2}}}$. 

We establish the following lower bound on the estimation error of the empirical median estimator.
\begin{theorem}
  \label{prop:Lower-bound-median}
  
    Let $\delta  \in \Big(\exp{(-(\sqrt{2}-1)^{2}n)},\frac{1}{4}\Big)$ and let $(X_1, \dots ,  X_n)$ be $n$ independent 
  Gaussian random variables where $X_i \sim \gaussdist (\theta,\sigma_i^2)$. Then, with probability at least $\delta$ we have
    \[ | \hat{\med}(X_1,\dots , X_n) - \theta | \ge 0.13\frac{\sqrt{n\ln{\left(\frac{1}{\delta}\right)}}}{\sum_{i=j+1}^{n}\sigma_{i}^{-1}},\]
    where $j= \lfloor0.05\sqrt{n\ln{\left(\frac{1}{\delta}\right)}}\rfloor$.
\end{theorem}
The proof of Theorem \ref{prop:Lower-bound-median} relies on Corollary \ref{cor:corollaire cas gaussien} derived from the following lemma, the proof of which is given in Section \ref{proof corollaire}.
\begin{lemma}
    \label{lemme minoration}
    Let $\delta \in (\exp{(-n)},1)$ and $S= \sum_{i=1}^{n} V_{i}$, where $V_1, \dots, V_n$ are independent random variables with $V_{i} \sim \mathcal{B}(p_{i})$ for all $i \in \lint 1,n \rint$.
    If $p_{i} \in (\frac{1}{4},\frac{3}{4})$ for every $i \in \lint 1,n \rint$, then
    \[\P\left(S \ge \mathbb{E}[S] + 0.3\sqrt{n\ln{\left(\frac{2}{\delta}\right)}}\right) \ge \frac{\delta}{2}.\]
\end{lemma}
We conclude that the lower bound on the estimation error of the empirical median estimator matches the upper bound within a numerical constant.
Noteworthy, the smallest $\sqrt{n}$ values of $\sigma_i$ do not influence the behavior of the median.

\section{Discussion and prior work}
\label{sec:discussion}

\paragraph*{Prior work}
In the context of estimating the common mean of heteroscedastic random variables, several research projects have been conducted, contributing to the understanding and development of robust estimators. For instance, the article by~\cite{devroye2023mean} investigates the estimation of the common mean of independent random variables with symmetric and unimodal densities. The authors propose an adaptive estimator that combines the empirical median estimator and the modal interval estimator, demonstrating its robustness in handling noise and outliers.

Similarly, the works by~\cite{chierichetti2014learning} and~\cite{pensia2019estimating} explore an estimator that integrates the empirical median and the $k$-shorth interval estimator. This $k$-shorth estimator identifies the center of the shortest interval containing at least $k$ points, leveraging both location and density properties of the data for increased efficiency. These approaches align with a growing interest in hybrid estimators that incorporate multiple statistical principles to balance robustness and precision.

Additionally to these estimators, the empirical median, a cornerstone in robust statistics, has been extensively studied for its desirable properties in the presence of contamination and heavy-tailed distributions. Foundational works, such as~\cite{hodges2011estimates}, introduced the median as an efficient and robust alternative to the mean for estimating central location in rank-based methods. Further theoretical refinements were explored in~\cite{chernoff1961sequential}, which analyzed sequential methods incorporating the median, and~\cite{koenker1978regression}, which connected median-based estimation to regression quantiles, offering a framework for robust regression models capable of handling heteroscedasticity and non-linearity.

In recent years, the empirical median has been further adapted to meet modern statistical challenges. Research by~\cite{catoni2012challenging} developed robust estimators based on median-like properties for heavy-tailed distributions, while~\cite{diakonikolas2017being} employed the median-of-means technique to enhance robustness under adversarial contamination. These methods highlight the median's versatility and utility in high-dimensional settings, where traditional estimators may fail. Additional theoretical refinements, as explored in~\cite{rousseeuw1986robust}, further solidify the median's role in robust statistics by addressing challenges such as ties and varying variances, underscoring its ongoing relevance in statistical inference.

Several foundational works laid the groundwork for robust estimation techniques. The seminal article by Huber~\cite{huber1964robust} introduced robust estimators that mitigate the impact of outliers, a principle further extended by Tukey's biweight estimator~\cite{tukey1977exploratory}. In the context of heteroscedasticity, ~\cite{brown1971admissible} analyzed weighted means and highlighted the limitations of traditional approaches when variances are only partially known. Building on median-based methods, ~\cite{pj1987robust} proposed the Least Median of Squares estimator, which is particularly effective in settings with non-homogeneous variances and outliers where~\cite{catoni2012challenging} introduced an estimator tailored for heavy-tailed distributions, leveraging advanced tail-bound techniques.

In addition, the article by~\cite{yuan2020learning} delves into truncation-based methods for mean estimation under heteroscedasticity. The authors introduce an iterative estimator designed for datasets with a subset of variances bounded by a known threshold, while allowing the remaining variances to be arbitrarily large. The proposed estimator iteratively removes outliers by truncating a proportion of points at each step and calculating the mean of the remaining subset. This method highlights the importance of adaptive approaches in handling data with heterogeneous variances and unknown contamination levels.

Moreover, the work of~\cite{xia2019non} sheds light on the relationship between the error bound of the empirical median and the harmonic mean of the standard deviations when estimating the common median of independent Gaussian random variables. This insight underscores the critical role of variance heterogeneity in determining estimator performance and paves the way for more refined approaches in the Gaussian setting.

A more recent study by~\cite{compton2024near} introduces a novel "balanced estimator". This estimator, a variation of modal estimators, imposes an additional condition: the number of observations in the right neighborhood of the estimator must approximately equal those in the left neighborhood. This balanced approach addresses asymmetries in data distributions and enhances estimator performance under specific conditions, showcasing the evolving landscape of robust mean estimation techniques.

Another line of work considers the estimation of a location parameter under symmetric noise with unknown distribution. This approach, introduced by Stone \cite{stone1975adaptive}, consists in first estimating the noise density using kernel methods and then performing a Newton step to maximize the resulting log-likelihood. This yields an estimator closely related to the maximum likelihood estimator. Building on this idea, \cite{gupta2023finite} provide finite-sample optimality results, and show that the estimation error is governed by the Fisher information of the unknown noise distribution. Their analysis highlights the central role of Fisher information in the symmetric location model, even in the absence of full parametric knowledge.

Finally, the question of whether classical minimax lower bounds can be matched by actual estimators in finite-sample regimes has recently been addressed by~\cite{compton2025attainability} in the context of location estimation. In their work, the authors investigate whether the standard Le Cam-type lower bounds, based on binary hypothesis testing, can be attained by estimators under finite-sample constraints. They show that for symmetric mixtures of log-concave densities, the minimax rate can indeed be achieved (up to logarithmic factors). However, they also prove that for the broader class of symmetric unimodal distributions, no estimator can universally attain the two-point testing lower bound.

Historical context is also relevant to this line of research. The method of medians in the theory of errors, as initially proposed by Kolmogorov~\cite{kolmogorov1931median} and revisited in various forms 
(\cite{kolmogoroff1941confidence},\cite{tukey1960survey},\cite{rousseeuw1986robust},\cite{pj1987robust}), provides a foundational perspective on robust estimators. These contributions have emphasized the advantages of median-based approaches in minimizing the influence of extreme values and ensuring stability across diverse distributions.

\paragraph{Comparison with previous results.}
The upper bound on the estimation error of the empirical median that we established improves and generalizes on the results of {Devroye, Lattanzi and Lugosi~\cite{devroye2023mean}} and {Xia~\cite{xia2019non}}. Indeed, our bound is applicable in a more general setting than that of 
{Xia~\cite{xia2019non}} and is more accurate than the one of {Devroye et al.~\cite{devroye2023mean}}. 

We can start by comparing our bound with that of ~\cite{xia2019non}, which studies the empirical median of $n$ independent but non-identically distributed random variables with a common median.
Xia's bound is as follows.
\begin{proposition} \text{\cite{xia2019non}}
   Let $\delta \in (0,1)$ and $X_i \sim \gaussdist (\theta,\sigma_i^2)$ where $\sigma_{1} > \frac{2\sqrt{2}}{0.35} \frac{\sqrt{n\ln{(\frac{2}{\delta})}}}{\sum_{i=1}^{n} \sigma_{i}^{-1}} $. Then, with probability at least $1-\delta$ we have
\[\lvert \hat{\med}(X_1,\dots , X_n) - \theta \rvert \le \frac{\sqrt{2}}{0.35} \frac{\sqrt{n\ln{(\frac{2}{\delta})}}}{\sum_{i=1}^{n} \sigma_{i}^{-1}} .\] 
\end{proposition}
The two bounds are of the same shape. However, the bound of \cite{xia2019non} is only valid under a specific condition on the $\sigma_i$'s, namely,
\begin{equation}
\label{eq:condition Xia}
    \sigma_{1} > \frac{2\sqrt{2}}{0.35} \frac{\sqrt{n\ln{(\frac{1}{\delta})}}}{\sum_{i=1}^{n} \sigma_{i}^{-1}}
\end{equation}
while our estimator ensures this bound for any configuration of $\sigma_{1}, \dots, \sigma_{n}$.
Furthermore, our bound reveals the fact that the lowest $\sqrt{n}$ $\sigma_i$'s do not influence the behavior of the median, while this fact is hidden when the condition (\ref{eq:condition Xia}) is imposed.

We then compare the estimation error of our estimator with that of the empirical median estimator studied in \cite{devroye2023mean} which does not make any assumption on $\sigma_1$. This is equivalent to comparing our bound obtained in Theorem \ref{prop:Upper-bound-median} to
\begin{equation}
\label{eq: condition Devroye}
   8e\sqrt{2}{\max}\left(\ln{\left(\frac{3}{\delta}\right)},\ln{(n+1)}\right)\beta^{-1}{\max}_{1\le j\le 8\sqrt{2n\ln({\frac{6}{\delta}})}}\frac{8\sqrt{2n\ln{\left(\frac{6}{\delta}\right)}}+1-j}{\sum_{i=j}^{n} \sigma_{i}^{-1}}. 
\end{equation}
where $\beta$ is a positive constant such that for all $t>0$, we have $\P(\lvert \frac{X_i-\theta}{\sigma_i} \rvert \ge t ) \ge \exp{(-\beta t)}$.
So, by taking
$j= \left\lfloor 4\sqrt{2n\ln{(\frac{6}{\delta})}} \right\rfloor$, we get that the bound (\ref{eq: condition Devroye}) is greater than \[ 32e\sqrt{2\pi}\ln{(n)}\frac{\sqrt{n\ln{(\frac{6}{\delta})}}}{\sum_{i=\left\lfloor4\sqrt{2n\ln{(\frac{6}{\delta})}}\right\rfloor+1}^{n}\sigma_{i}^{-1}}.\]
which includes an additional logarithmic factor compared to our bound.
Furthermore, the term under the maximum in (\ref{eq: condition Devroye}) corresponding to $j=1$ yields the bound of \cite{xia2019non}.
But again, given that it is a maximum over $j$, our bound performs better overall. Moreover, the bound of \cite{devroye2023mean} requires, in addition to the symmetry of the random variables, some tail assumptions on the distribution. In contrast, our bound is valid in a more general case requiring only the boundedness of the density around $0$ of the random variables. Thus, our result is more accurate but also applicable in a more general setting than the one of \cite{devroye2023mean}. 
\paragraph*{Comparison with the empirical mean estimator.}
We show now that the upper bound of the empirical median estimator is sharper than the lower bound of the empirical mean estimator which is given in Section \ref{sec:introduction}.
We want to show that 
\[\frac{\sqrt{\Big(\sum_{i=1}^{n}\sigma^2_{i}\Big)}}{n} \ge
\frac{\sqrt{n}}{\sum_{i=j+1}^{n}\sigma_{i}^{-1}}.\]
For that, we apply the following lemma.
\begin{lemma}
Let $\sigma_1 \le \sigma_2 \le \dots \le \sigma_n$ be positive real numbers. Then, for any integer $j \le \lfloor n/2 \rfloor - 1$, we have
\[
\frac{\sqrt{\sum_{i=1}^{n} \sigma_i^2}}{n} \ge c \cdot \frac{\sqrt{n}}{\sum_{i=j+1}^{n} \sigma_i^{-1}},
\]
for some absolute constant $c > 0$.
\end{lemma}

\begin{proof}
Let $k := \lfloor n/2 \rfloor$. At least $n - k \ge n/2$ of the $\sigma_i$ satisfy $\sigma_i \ge \sigma_k$, then
\[
\sum_{i=1}^n \sigma_i^2 \ge \sum_{i=k+1}^n \sigma_i^2 \ge (n - k) \sigma_k^2 \ge \frac{n}{2} \sigma_k^2,
\]
which implies that
\[
\frac{\sqrt{\sum_{i=1}^n \sigma_i^2}}{n} \ge \frac{\sqrt{n \sigma_k^2 / 2}}{n} = \frac{\sigma_k}{\sqrt{2n}}.
\]
On the other hand, since the sequence is non-decreasing and $j+1 \le k$, we have that 
$\sigma_{j+1} \le \sigma_k$ which implies that $\sigma_{j+1}^{-1} \ge \sigma_k^{-1}$.
There are $n - j$ terms in the sum \(\sum_{i = j+1}^n \sigma_i^{-1}\), and each is at most \(\sigma_{j+1}^{-1}\), hence
\[
\sum_{i = j+1}^n \sigma_i^{-1} \le (n - j)\, \sigma_{j+1}^{-1} \le n\, \sigma_{j+1}^{-1}.
\]
Thus
\[
\frac{\sqrt{n}}{\sum_{i=j+1}^n \sigma_i^{-1}} \le \frac{\sqrt{n}}{n \sigma_{j+1}^{-1}} = \frac{\sigma_{j+1}}{\sqrt{n}} \le \frac{\sigma_k}{\sqrt{n}}.
\]
Combining both bounds, we obtain that
\[
\frac{\sqrt{\sum_{i=1}^n \sigma_i^2}}{n} \ge \frac{\sigma_k}{\sqrt{2n}} \ge \frac{1}{\sqrt{2}} \cdot \frac{\sigma_k}{\sqrt{n}} \ge \frac{1}{\sqrt{2}} \cdot \frac{\sqrt{n}}{\sum_{i=j+1}^n \sigma_i^{-1}}.
\]
Therefore, the claimed inequality holds with constant \(c = 1/\sqrt{2}\).
\end{proof}

\paragraph*{Extension of the result on the lower bound to non Gaussian random variables.}
In this paper, we established an upper bound on the estimation error of the empirical median that is also valid in a more general setting than that of Gaussian random variables. In fact, this is also the case for our lower bound result. In fact, the only argument related to the Gaussian setting is that for all $i \in \lint 1,n \rint$, $\P\left(\theta \le X_{i} \le \theta+t\right) \le \frac{t}{\sigma_{i}\sqrt{2\pi}}.$
Thus, this result on the lower bound extends to random variables with bounded densities around their median $\theta$. Thus, we can even consider random variables with infinite moments or also let the distribution have heavy tails, which was not covered by \cite{devroye2023mean}. 

\paragraph*{Limits of the empirical median estimator.}
The empirical median estimator is not optimal in a general heteroscedastic framework, as it fails to exploit the information carried by the $\sqrt{n}$ smallest variances. Our bounds reflect this limitation, since they remain unaffected by these variances. This shows that, although the empirical median is robust and simple, it cannot reach classical minimax optimality in such settings. Recently, this broader question has been clarified in the context of location estimation by~\cite{compton2025attainability}, who investigated whether classical minimax lower bounds can be attained by finite-sample estimators. They showed that, while minimax rates can indeed be achieved (up to logarithmic factors) for symmetric mixtures of log-concave densities, no estimator can universally attain the two-point testing lower bound for the broader class of symmetric unimodal distributions. Altogether, our findings illustrate the fundamental limits of the empirical median and highlight open directions for developing estimators that can match minimax lower bounds in general frameworks.

\section{Proofs}
\label{sec:proofs}
\paragraph{Assumption}
Throughout the proofs, we assume that $\theta=0$. This is without loss of generality since the median is invariant by translation.

\subsection{Proof of Theorem \ref{prop:Upper-bound-median}}
\begin{proof}
\label{proof upper bound}
Let $t\ge 0$. 
We want to control \[\P\Big(\hat{\med}(X_1,\dots , X_n) \ge t\Big).\]
We define
\[D(t)= \sum_{i=1}^{n} \Big(\mathbb{1}_{X_{i} \ge t}-\P(X_{i} > t)\Big) \quad \text{and} \quad B(t) =\sum_{i=1}^{n} \P(0 \le X_{i} \le t).\]
By Lemma \ref{lemma 2}, we know that $\hat{\med}(X_1,\dots , X_n) \le t$ is equivalent to  $D(t) \le B(t)$.
We thus start by bounding from above $D(t)$.
By Hoeffding's inequality {(\cite[Theorem~2.8 p.~34]{boucheron2013concentration})}, with probability at least equal to $1-\delta$ we have
\[D(t) \le \sqrt{\frac{1}{2}} \sqrt{n\ln{\left(\frac{1}{\delta}\right)}}.\]
The objective is now to bound from below $B(t)$.
We set $ k_{t} = \sup {(i: t> \sigma_{i} )}$.
We fix $C$ a positive constant such that $0<C<1$. By assumption of the theorem, for all $i \in \lint 1,n \rint$ there exists $\sigma_i$ such that if $0 \le t \le \sigma_i$ then $\P(0 \le X_i \le t) \ge C \frac{t}{\sigma_i}$.
Then,
\begin{align*}
    B(t)=\sum_{i=1}^{n} \P(0 \le X_{i} \le t) &= \sum_{i=1}^{k_{t}} \P(0 \le X_{i} \le t) + \sum_{i=k_{t}+1}^{n} P(0 \le X_{i} \le t) \nonumber\\
        &\ge C k_{t} + C t \sum_{i=k_{t}+1}^{n} \sigma_{i}^{-1} .\nonumber
\end{align*}
Therefore, $\hat{\med}(X_1,\dots , X_n) \le t$ as soon as $B(t) \ge D(t)$ which is satisfied with probability at least $1- \delta$ if 
\begin{equation}
    \label{eq:condition on t general}
    k_{t} + t \sum_{i=k_{t}+1}^{n} \sigma_{i}^{-1} \ge \frac{1}{C\sqrt{2}} \sqrt{n\ln{\left(\frac{1}{\delta}\right)}}.
\end{equation}
We have that
\[ j := \left\lfloor \frac{1}{C\sqrt{2}} \sqrt{n\ln{\left(\frac{1}{\delta}\right)}} \right\rfloor < n,\]
under the assumption that $\delta > \exp{\left(-2nC^2\right)}$.
We then define
\[ t=t^{*}:= \frac{ \sqrt{n\ln{\left(\frac{1}{\delta}\right)}}}{C\sqrt{2}\sum_{i=j+1}^{n} \sigma_{i}^{-1}},\]
and we show that this choice of $t$ satisfies condition (\ref{eq:condition on t general}).
Two distinct cases may be identified.
\begin{itemize}
    \item If $k_{t} \ge j+1$, then (\ref{eq:condition on t general}) is verified for all $t$.
    \item If $k_{t} \le j$, then 
    \[t^{*} \sum_{i=k_{t}+1}^{n} \sigma_{i}^{-1} \ge t^{*} \sum_{i=j+1}^{n} \sigma_{i}^{-1}= \frac{1}{C\sqrt{2}} \sqrt{n\ln{\left(\frac{1}{\delta}\right)}}.\]
    It is therefore true that (\ref{eq:condition on t general}) is satisfied.
\end{itemize}
Thus, with probability at least $1-\frac{\delta}{2}$, we have 
\[  \hat{\med}(X_1,\dots , X_n) \le \frac{\sqrt{n\ln{\left(\frac{2}{\delta}\right)}}}{C\sqrt{2}\sum_{i=j+1}^{n}\sigma_{i}^{-1}} .\]
Consequently, with probability at least equal to $1-\delta$, 
\[ \lvert \hat{\med}(X_1,\dots , X_n) \rvert \le \frac{\sqrt{n\ln{\left(\frac{2}{\delta}\right)}}}{C\sqrt{2}\sum_{i=j+1}^{n}\sigma_{i}^{-1}} .\]
\end{proof}
\subsection{Proof of Theorem \ref{prop:Lower-bound-median}}
\begin{proof}
\label{proof lower bound}
In this proof we will use the following corollary which is an application of the Lemma \ref{lemme minoration} in the Gaussian case.
\begin{corollary}
\label{cor:corollaire cas gaussien}
Let $\delta \in (\exp{(-n)},1)$, $t>0$ and $(X_1, \dots ,  X_n)$ be $n$ independent 
  Gaussian random variables where $X_i \sim \gaussdist (\theta,\sigma_i^2)$. We define $S= \sum_{i=1}^{n} V_{i}$ such that $V_{i}= \mathbb{1}_{X_{i} \ge t}$ and $p_{i}(t)= \P(X_{i} \ge t)$.
    If $t \le 0.63\sigma_{1}$, then
    \[\P\left(S \ge \mathbb{E}[S] + 0.3\sqrt{n\ln{\left(\frac{2}{\delta}\right)}}\right) \ge \frac{\delta}{2}.\]
\end{corollary}
The corollary follows immediately from Lemma \ref{lemme minoration} and the fact that under assumption of Corollary \ref{cor:corollaire cas gaussien} we have that for all $i \in \lint 1,n \rint$, $p_{i}(t) \in [\frac{1}{4},\frac{3}{4}]$.
We want to bound from below the probability 
 \[\P\Big(\hat{\med}(X_1,\dots , X_n) \ge t\Big).\]
According to the Lemma \ref{lemma 2}, we have that $\hat{\med}(X_1,\dots , X_n) \ge t$ is equivalent to $D(t) \ge B(t)$.
We want to control $D(t)$ and $B(t)$.
We start by bounding from above $B(t)$. To do this, we have
\begin{align*}      
      \P(0 \le X_{i} \le t) &= \int_{0}^{t} \frac{1}{\sqrt{2\pi\sigma_{i}^2 }} \exp{(-\frac{x^2}{2\sigma_{i}^2})} \,\di x \\
      &\leq \int_{0}^{t} \frac{1}{\sqrt{2\pi\sigma_{i}^2 }} \,\di x 
      =  \frac{t}{\sigma_{i}\sqrt{2\pi}} \nonumber .
    \end{align*}
We then have that 
\begin{align*}
    B(t) &=  \sum_{i=1}^{k_{t}}\P(0 \le X_{i} \le t)+ \sum_{i=k_{t}+1}^{n}\P(0 \le X_{i} \le t)  \\
    &\le \sum_{i=1}^{k_{t}}1+ \sum_{i=k_{t}+1}^{n}\frac{t}{\sigma_{i}\sqrt{2\pi}} \nonumber = k_{t}+\frac{t}{\sqrt{2\pi}}\sum_{i=k_{t}+1}^{n}\sigma_{i}^{-1} \nonumber .
\end{align*} 
The objective is now to bound from below $D(t)$.
We have
\[D(t)= \sum_{i=1}^{k_{t}} \Big(\mathbb{1}_{X_{i} \ge t}-\P(X_{i} > t)\Big)+ \sum_{i=k_{t}+1}^{n} \Big(\mathbb{1}_{X_{i} \ge t}-\P(X_{i} > t)\Big).\]
If $i \le k_{t}$, which is equivalent to $t>\sigma_{i}$, it follows that $\P(X_{i}>t) = \P(\frac{X_{i}}{\sigma_i} > \frac{t}{\sigma_i}) \le \frac{1}{4}$. 
Thus, by Hoeffding's inequality applied to the first sum, we have \[\P\left(\sum_{i=1}^{k_{t}} \Big(\mathbb{1}_{X_{i} \ge t}-\P(X_{i}>t)\Big) \ge -\frac{1}{\sqrt{2}}\sqrt{k_{t}\ln{\left(\frac{1}{\delta}\right)}}\right) \ge 1-\delta .\]
Next, by Corollary \ref{cor:corollaire cas gaussien} we obtain 
\[ \P\left(\sum_{i=k_{t}+1}^{n} \Big(\mathbb{1}_{X_{i} \ge t}-\P(X_{i}>t)\Big) \ge (c-\frac{1}{\sqrt{2}})\sqrt{(n-k_{t})\ln{\left(\frac{1}{2\delta}\right)}}\right) \ge 2\delta,\]
with $c-\frac{1}{\sqrt{2}} \le 0.3$ and $\delta \le \frac{1}{4}$ such that $\ln{(\frac{1}{2\delta})} \ge \frac{1}{2}\ln{(\frac{1}{\delta})}$.
So, with probability at least $\delta$, we have that
\begin{align*}
        D(t) &\ge -\frac{1}{\sqrt{2}}\sqrt{k_{t}\ln{\left(\frac{1}{\delta}\right)}}+ \left(c-\frac{1}{\sqrt{2}}\right)\sqrt{(n-k_{t})\ln{\left(\frac{1}{\delta}\right)}}  \\
        &= \Big(\left(c-\frac{1}{\sqrt{2}}\right)\sqrt{n-k_{t}}-\frac{1}{\sqrt{2}}\sqrt{k_{t}}\Big)\sqrt{\ln{\left(\frac{1}{\delta}\right)}}. \nonumber
\end{align*}
Since $c\le \frac{1}{\sqrt{2}}+0.3$, by taking $c=1$ we obtain with probability at least $\delta$ that 
\[ D(t) \ge \Big(\left(1-\frac{1}{\sqrt{2}}\right)\sqrt{n-k_{t}}-\frac{1}{\sqrt{2}}\sqrt{k_{t}}\Big)\sqrt{\ln{\left(\frac{1}{\delta}\right)}}.\]
Now, we want to fix $k_{t}$ such that
\begin{equation}
\label{eq:admissible}
    k_{t} \le \frac{n}{2} \quad \text{and } \quad k_{t} \le \frac{n(\sqrt{2}-1)^2}{8},
\end{equation}
which will imply that $\left
(1-\frac{1}{\sqrt{2}}\right)\sqrt{n-k_{t}}-\frac{1}{\sqrt{2}}-\sqrt{k_{t}} \ge  \frac{\sqrt{2}-1}{2\sqrt{2}}\sqrt{\frac{n}{2}}$
and we thus obtain that, with probability at least $\delta$, that
\[D(t) \ge \frac{\sqrt{2}-1}{4}\sqrt{n\ln{\left(\frac{1}{\delta}\right)}}.\] 
Subsequently, we have that, with probability at least $\delta$,
\[\hat{\med}(X_1,\dots , X_n) \ge t \text{ is verified if  }\frac{\sqrt{2}-1}{4}\sqrt{n\ln{\left(\frac{1}{\delta}\right)}} \ge k_{t}+\frac{t}{\sqrt{2\pi}}\sum_{i=k_{t}+1}^{n}\sigma_{i}^{-1}.\]
We choose $k_{t}$ such that 
\[k_{t}= \left\lfloor \frac{1}{2}\frac{\sqrt{2}-1}{4}\sqrt{n\ln{(\frac{1}{\delta})}} \right\rfloor = \left\lfloor \frac{\sqrt{2}-1}{8}\sqrt{n\ln{(\frac{1}{\delta})}} \right\rfloor ,\]
We first check that this value of $k_{t}$ is admissible meaning that it satisfies (\ref{eq:admissible}).
Since $\frac{n(\sqrt{2}-1)^2}{8} \le \frac{n}{2}$, this is equivalent to verifying that $\frac{\sqrt{2}-1}{8}\sqrt{n\ln{\left(\frac{1}{\delta}\right)}} \le \frac{n(\sqrt{2}-1)^2}{8}.$
This condition is satisfied if \[\delta > \exp{\left(-(\sqrt{2}-1)^{2}n\right)}.\] 
We then have that $\hat{\med}(X_1,\dots , X_n) \ge t$ with probability at least $\delta$ as soon as 
\[\frac{\sqrt{2}-1}{8}\sqrt{n\ln{\left(\frac{1}{\delta}\right)}} \ge \frac{t}{\sqrt{2\pi}}\sum_{i=k_{t}+1}^{n}\sigma_{i}^{-1},\]
which is equivalent to
\[ t \le \frac{2-\sqrt{2}}{8}\sqrt{\pi}\frac{\sqrt{n\ln{\left(\frac{1}{\delta}\right)}}}{\sum_{i=k_{t}+1}^{n}\sigma_{i}^{-1}}.\]
We can therefore conclude that for all $ \delta \in \left(\exp{\left(-(\sqrt{2}-1)^{2}n\right)},\frac{1}{4}\right) $, we have with probability at least equal to $\delta$ that
\[ \hat{\med}(X_1,\dots , X_n) \ge \frac{2-\sqrt{2}}{8}\sqrt{\pi}\frac{\sqrt{n\ln{(\frac{1}{\delta})}}}{\sum_{i=k_{t}+1}^{n}\sigma_{i}^{-1}},\]
with $k_{t}= \left\lfloor \frac{\sqrt{2}-1}{8}\sqrt{n\ln{(\frac{1}{\delta})}} \right\rfloor$.
\end{proof}
\subsection{Proof of Lemma \ref{lemme minoration}}
\begin{proof}

\label{subsec:preuve lemme minoration}
Let $t>0$ and $S=\sum_{i=1}^{n} V_{i}$
with $V_{i} \sim \mathcal{B}({p_{i}}) $, for $i \in \lint 1,n \rint$.
We consider a change of measure meaning that we move from $S$, a sum of random variables with Bernoulli distribution of parameter $p_{i}$ to $\Tilde{S}$, a sum of random variables with Bernoulli distribution of parameter $\Tilde{p_{i}}$ such that 
\[\Tilde{p_{i}}=p_{i}+ \Delta = p_{i}+ c\sqrt{\frac{\ln{(\frac{1}{\delta})}}{n}} , \quad \Delta=c\sqrt{\frac{\ln{\left(\frac{1}{\delta}\right)}}{n}} , \quad \delta > \exp{(-n)},\]
where $c$ is a positive constant to be determined later.
So we have that 
\[\mathbb{E}[\Tilde{S}] = \mathbb{E}[S]+ c\sqrt{n\ln{\left(\frac{1}{\delta}\right)}}.\]
The advantage of this change of measure is that we end up with new random variables whose expectation is large enough to reduce the probability of deviation but whose distribution function remains close to that of the initial random variables.

We assume that for all $i \in \lint 1,n \rint$, $p_{i} \in (\frac{1}{4},\frac{3}{4})$. 
Let us show that $\tilde{p_{i}}$ is also in the interval $\left(\frac{1}{4},\frac{3}{4}\right)$ for all $i \in \lint 1,n \rint$.
We know that $\tilde{p_{i}}= p_{i}+ \Delta$, so $\tilde{p_{i}} > \frac{1}{4}.$
And since $p_{i} < \frac{3}{4}$, there exists $\epsilon$ such that $p_{i}=\frac{3}{4}-\epsilon$. We deduce that 
$\tilde{p_{i}} < \frac{3}{4}$ as long as
$c < \eps\sqrt{\frac{n}{\ln{(\frac{1}{\delta})}}}$
that is verified for $n$ sufficiently large. Since $p_{i}$ and $\Tilde{p_{i}}$ are of order $\frac{1}{2}$, the Hoeffding bound is optimal.
Now, let $c' \in \mathbb{R}$ and we want to bound from below \[\P\left(S-\mathbb{E}[S] \ge c'\sqrt{n\ln{\left(\frac{1}{\delta}\right)}}\right). \]
For that, we need to bound from below
$ \P\left(\Tilde{S} \ge \mathbb{E}[S] + c'\sqrt{n\ln{\left(\frac{1}{\delta}\right)}}\right), $ and to bound from above $\lvert \P(\Tilde{S} \ge t)- \P(S \ge t) \rvert.$
We start by reducing from below $\P\left(\Tilde{S} \ge \mathbb{E}[S] + c'\sqrt{n\ln{\left(\frac{1}{\delta}\right)}}\right).$
We have that
\begin{align*}
     \P\left(\Tilde{S} \ge \mathbb{E}[S] + c'\sqrt{n\ln{\left(\frac{1}{\delta}\right)}}\right) 
        &=  \P\left(\Tilde{S} \ge \mathbb{E}[\Tilde{S}]-c\sqrt{n\ln{\left(\frac{1}{\delta}\right)}} + c'\sqrt{n\ln{\left(\frac{1}{\delta}\right)}}\right) \\
        &= \P\left(\Tilde{S} \ge \mathbb{E}[\Tilde{S}]-(c-c')\sqrt{n\ln{\left(\frac{1}{\delta}\right)}} \right).
\end{align*}       
We define $c''=c-c' >0$ and we want to show that
\[\P\left(\Tilde{S} \ge \mathbb{E}[\Tilde{S}]-c''\sqrt{n\ln{\left(\frac{1}{\delta}\right)}} \right) \ge 1- \frac{\delta}{2}.\]
This is equivalent to showing that
\[\P\left(\Tilde{S} \le \mathbb{E}[\Tilde{S}]-c''\sqrt{n\ln{\left(\frac{1}{\delta}\right)}} \right) \le \frac{\delta}{2}. \]
Moreover, $-\Tilde{S}= \sum_{i=1}^{n}(-\Tilde{V_{i}}) $ where $(-\Tilde{V_{i}})$ are $\frac{1}{8}$-Sub-Gaussian random variables.
By Hoeffding's inequality, we then have that
\[\P\left(-\Tilde{S}-\mathbb{E}[-\Tilde{S}] \ge \frac{1}{\sqrt{2}}\sqrt{n\ln{\left(\frac{2}{\delta}\right)}} \right) \le \frac{\delta}{2}, \]
which implies that 
$\P\left(\Tilde{S} \ge \mathbb{E}[\Tilde{S}]-\frac{1}{\sqrt{2}}\sqrt{n\ln{\left(\frac{2}{\delta}\right)}} \right) \ge 1- \frac{\delta}{2}$ and thus we obtain that $c''= \frac{1}{\sqrt{2}}$.
Consequently, we obtain that 
\[\P\left(\Tilde{S} \ge \mathbb{E}[S] + \left(c-\frac{1}{\sqrt{2}}\right)\sqrt{n\ln{\left(\frac{2}{\delta}\right)}}\right) \ge 1-\frac{\delta}{2}, \quad \text{and } c'= c-\frac{1}{\sqrt{2}} >0.\]
We then want to bound from above $\lvert \P(\Tilde{S} \ge t)- \P(S \ge t) \rvert.$
We have that
\begin{align*}
    \lvert \P(\Tilde{S} \ge t)- P(S \ge t) \rvert &\le \underset{A \in \mathbb{B(\R)}}{\sup}{\lvert \P(\Tilde{S} \in A)- \P(S \in A) \rvert}  \\
        &= \underset{A \in \mathbb{B(\R)}}{\sup}{\lvert \P_{\Tilde{S}}(A)-\P_{S}(A)  \rvert}.
\end{align*}
We define $P= \P_{S}$ and $Q=\P_{\Tilde{S}}$ such that
\[S=\Phi(X_{1},\dots,X_{n})=X_{1}+\dots+ X_{n}, \quad  X_{i} \sim \mathcal{B}({p_{i}}) \] 
and
\[\Tilde{S}=\Phi(Y_{1},\dots,Y_{n})=Y_{1}+\dots+ Y_{n},  \quad Y_{i} \sim \mathcal{B}(\Tilde{{p_{i}}}).\]
We have that
\[\lvert \P(\Tilde{S} \ge t)- \P(S \ge t) \rvert \le \dtv(P,Q), \quad  \dtv(P,Q)= \underset{A \in \mathbb{B(\R)}}{\sup}{\lvert P(A)-Q(A) \rvert}.\]
By \cite[(2.25) p.89]{tsybakov2009nonparametric}, we have that  
$\dtv(P,Q) \le 1- \frac{1}{2}\exp{(-\kll{P}{Q})}$ and by Lemma $2.7$ in \cite{tsybakov2009nonparametric}, we have that  
$\kll{P}{Q} \le \log{(1+\chi^2(P,Q))}$
where $\chi^2$ denotes the $\chi^2$ divergence.
Besides, for $P_{i},Q_{i}, i \in 1,\dots,n$ laws on $E_{i}$ and for $P=\bigotimes _{i=1}^n P_i$ and $Q=\bigotimes _{i=1}^n Q_i$ product laws on $\prod_{i=1}^{n} E_{i}$, we have that   $1+\chi^2(P,Q)= \prod_{i=1}^{n}(1+\chi^2(P_i,Q_i))$.
Thus, 
\[\kll{P}{Q} \le n(1+\chi^2(P_i,Q_i)) \le n\chi^2(P_i,Q_i).\]
Since $p_i,\tilde{p_i} \in [\frac{1}{4},\frac{3}{4}]$, we have that 
$\chi^2(P_i,Q_i)= \frac{(p_i-\tilde{p_i})^2}{\tilde{p_i}} \le 4(p_i-\tilde{p_i})^2=4c^2\frac{\ln\left({\frac{1}{\delta}}\right)}{n}  .$
We then obtain that 
$\kll{P}{Q} \le 4c^2\ln\left({\frac{1}{\delta}}\right)$.
So, 
\begin{align*}
    \dtv(P,Q) &\le 1- \frac{\exp{(-4c^2\ln({\frac{1}{\delta}}))}}{2} 
        = 1- \frac{\exp{(4c^2\ln({\delta})})}{2}. 
\end{align*} 
We want to have $\dtv(P,Q) \le 1-\delta$. To do this, we let $c > \sqrt{\frac{\ln{(2\delta)}}{4\ln{(\delta)}}} \ge \frac{1}{2}$. 
We conclude that 
\[\lvert P(\Tilde{S} \ge t)- P(S \ge t) \rvert \le 1- \delta.\]
Finally, by combining the two following results
\begin{itemize}
    \item $P\left(\Tilde{S} \ge \mathbb{E}[S] + (c-\frac{1}{\sqrt{2}})\sqrt{n\ln{(\frac{2}{\delta})}}\right) \ge 1-\frac{\delta}{2}$,
    \item $\lvert P(\Tilde{S} \ge t)- P(S \ge t) \rvert \le 1- \delta$,
\end{itemize}
we obtain that 
\begin{align*}
    \P\left(S \ge \mathbb{E}[S] + (c-\frac{1}{\sqrt{2}})\sqrt{n\ln{(\frac{2}{\delta})}}\right) &\ge \P\left(\Tilde{S} \ge \mathbb{E}[S] + (c-\frac{1}{\sqrt{2}})\sqrt{n\ln{(\frac{2}{\delta})}}\right) \\
 &- \Big\lvert \P\left(S \ge \mathbb{E}[S] + (c-\frac{1}{\sqrt{2}})\sqrt{n\ln{(\frac{2}{\delta})}}\right)\\
& - \P\left(\Tilde{S} \ge \mathbb{E}[S] + (c-\frac{1}{\sqrt{2}})\sqrt{n\ln{(\frac{2}{\delta})}}\right) \Big\rvert,
\end{align*}
 and so that
    $ \P\left(S \ge \mathbb{E}[S] + (c-\frac{1}{\sqrt{2}})\sqrt{n\ln{(\frac{2}{\delta})}}\right) \ge \frac{\delta}{2}$
for any $c > {\max}\left(\frac{1}{2},\frac{1}{\sqrt{2}}\right) > \frac{1}{\sqrt{2}}$. Taking $c=1$, we obtain that 
\[\P\left(S \ge \mathbb{E}[S] + 0.3\sqrt{n\ln{(\frac{2}{\delta})}}\right) \ge \frac{\delta}{2} .\]
\end{proof}
\subsection{Elementary lemmas and proofs}
In this section, we provide for the sake of completeness statements and short proofs of the facts used.
\begin{lemma}
    \label{lemma 2}
     Let $t \ge 0$, then 
    \[\hat{\med}(X_1,\dots , X_n) \ge t \quad \text{if and only if} \quad \sum_{i=1}^{n} \Big(\mathbb{1}_{X_{i} \ge t}-\P(X_{i}>t)\Big) \ge \sum_{i=1}^{n} \P(0 \le X_{i} \le t).\]
\end{lemma}
\begin{proof}
     We have that $\hat{\med}(X_1,\dots , X_n)\ge t$ is equivalent to
    $\sum_{i=1}^{n} \mathbb{1}_{X_{i} \ge t} \ge \frac{n}{2}$
    which is also equivalent to 
    \[\sum_{i=1}^{n} \left(\mathbb{1}_{X_{i} \ge t}-\P(X_{i}>t)\right) \ge \frac{n}{2}- \sum_{i=1}^{n}\P(X_{i}>t)\]
    or also
    \[\sum_{i=1}^{n} \left(\mathbb{1}_{X_{i} \ge t}-\P(X_{i}>t)\right) \ge \sum_{i=1}^{n} \P(0 \le X_{i} \le t) \, . 
    \qedhere
    \]
\end{proof}

\paragraph{Proof of Corollary \ref{cor:Upper-bound-median}}
\begin{proof}

Let $t \ge 0$. It suffices to find the constant $C$ satisfying that for all $i \in \lint 1,n \rint$ there exists $\sigma_i$ such that if $0 \le t \le \sigma_i$ then $\P(0 \le X_i \le t) \ge C \frac{t}{\sigma_i}$ and then apply the same proof of Theorem~\ref{prop:Upper-bound-median}.
If $t \le  \sigma_{i}$,
        then
            \begin{align*}
                \P(0\le X_{i} \le t) &=\int_{0}^{t} \frac{1}{\sqrt{2\pi\sigma_{i}^2 }} \exp{\left(-\frac{x^2}{2\sigma_{i}^2}\right)} \,\di x \\
                 \ge& \int_{0}^{t} \frac{1}{\sqrt{2\pi\sigma_{i}^2 }} \exp{(-\frac{1}{2})} \,\di x  = \frac{\exp{(-\frac{1}{2})}}{\sqrt{2\pi}} \frac{t}{\sigma_{i}}.
            \end{align*} 
We therefore have that with probability at least equal to $1-\delta$, 
\[ \lvert \hat{\med}(X_1,\dots , X_n) \rvert \le \frac{\sqrt{\pi}\exp{(2)}\sqrt{n\ln{\left(\frac{2}{\delta}\right)}}}{\sum_{i=j+1}^{n}\sigma_{i}^{-1}} .\]
\end{proof}

\paragraph{Proof of Corollary 
\ref{cor:corollaire cas gaussien}}
\begin{proof}
\label{proof corollaire}
We apply Lemma \ref{lemme minoration} to 
$S= \sum_{i=1}^{n} V_{i}= \sum_{i=1}^{n} \mathbb{1}_{X_{i} \ge t} $
 with $V_{i} \sim \mathcal{B}({p_{i}(t)}) $ and $p_{i}(t)= \P(X_{i} \ge t)$.
 For that, we just need to verify that if $t \le 0.63\sigma_{1}$ we have that for all $i \in \lint 1,n \rint$, $p_{i}(t) \in [\frac{1}{4},\frac{3}{4}]$.
We have that 
$p_{i}(t)= \P(X_{i} \ge t)= \psi(\frac{t}{\sigma_{i}}),$ where $\psi$ is the cumulative distribution function of a standard Gaussian random variable.
Since $t>0$, we have that $p_{i}(t) \le \frac{1}{2} \le \frac{3}{4}$.
Furthermore, $p_{i}(t) \ge \frac{1}{4}$. Indeed, 
\[p_{i}(t)= \P(X_{i} \ge t)= \P\left(X \ge \frac{t}{\sigma_{i}}\right) \ge \P\left(X \ge \frac{t}{\sigma_{1}}\right) \ge \frac{1}{2}- \frac{t}{\sigma_{1}}\frac{1}{\sqrt{2\pi}},\]
where $X$ is a standard Gaussian random variable.
Therefore, if 
\[t \le \frac{\sqrt{2\pi}}{4}\sigma_{1} \le 0.63\sigma_{1},\] 
we have that $p_{i}(t) \ge \frac{1}{4}$ for all $i \in \lint 1,n \rint$.
\end{proof}

\paragraph{Acknowledgments}
I would like to express my gratitude to my two PhD supervisors for their contribution to the realisation of this work. I am indebted to Jaouad Mourtada for his crucial assistance in conceiving the ideas and elaborating the proofs, and to Alexandre Tsybakov for meticulously proofreading the manuscript and offering invaluable advice.
\appendix

\bibliographystyle{abbrv}
\bibliography{biblio}
\end{document}